\documentclass{article}
\usepackage{amsfonts}
\usepackage{amsmath}

\usepackage{amssymb}
\usepackage{graphicx}
\usepackage{latexsym}
\usepackage{color}
\usepackage{mathrsfs}

\newcommand{\di}{\mathrm{d}}

\newtheorem{theorem}{Theorem}

\newtheorem{definition}[theorem]{Definition}
\newtheorem{example}[theorem]{Example}

\newtheorem{lemma}[theorem]{Lemma}

\newtheorem{proposition}[theorem]{Proposition}
\newtheorem{remark}[theorem]{Remark}

\newenvironment{proof}[1][Proof]{\noindent\textbf{#1.} }{\ \rule{0.5em}{0.5em}}

\begin{document}

\title{Regular Foliations and Poisson Structures on Orientable Manifolds}
\author{Ruben Flores-Espinoza$^1$ \\
Misael Avenda\~{n}o-Camacho$^2$  \\
\\
$^1$ Departamento de Matem\'aticas,\\
Universidad de Sonora, \\
Blvd L. Encinas y Rosales s/n Col. Centro, CP 83000,\\
Hermosillo (Son), Mexico.\\
\\
$^2$ CONACYT Research Fellow-Departamento de Matem\'aticas,\\
 Universidad de Sonora. \\}

%\institute{ Universidad de Sonora, Departamento de Matematicas.           }

\date{ }

\maketitle

\begin{abstract}
\noindent On an orientable manifold $M$,  we consider a regular even dimensional foliation $\mathcal{F}$ which is globally defined by a set of $k$-independent 1-forms. We give necessary and sufficient conditions for the existence of a regular Poisson structure on $M$ whose Characteristic foliation is precisely $\mathcal{F}$. Moreover, introducing a special class of the foliated 1-cohomology we describe obstructions for the existence of unimodular Poisson structures with a given characteristic foliation. In the same lines, we also give conditions for the existence of transversally constant Poisson structures.
\\

\noindent \textbf{Keywords} Regular foliation $\cdot$ Poisson tensor $\cdot$ Orientable manifolds $\cdot$ Foliated cohomology 
\\

\noindent \textbf{Mathematical Subject Classification (2000)} 	53D17 $\cdot$ 53D05 $\cdot$ 53C12
\end{abstract}

\section{Introduction}

The problem of existence of Poisson structures having a given regular even
dimensional foliation as its characteristic foliation has been considered by
different authors \cite{bert,GMP,HMS}. Here, using the foliated cohomology theory for
regular foliations and an approach to Poisson geometry on orientable
manifolds based on the use of differential forms and the application of the
trace operator calculus, we deal with the problem above and give necessary
and sufficient conditions for the existence of Poisson structures on
orientable manifolds for regular foliations globally defined by a set of k
independent 1-forms. Moreover, introducing a special class of the foliated
1-cohomology, we described obstructions for the existence of unimodular
Poisson structures with a given characteristic foliation. We also include in
the same lines, conditions for the existence of transversally constant
Poisson structures.

In section 2, we introduce, for a regular foliation and following \cite{HMS}, the complex of foliated forms and their foliated cohomology. Moreover, using the ideas in \cite{GMP} we introduce an invariant called
the first obstruction class belonging to the foliated 1-cohomology. When
this invariant vanishes the foliation can be defined by a decomposable
closed form. In the section 3, the basic geometrical objects of Poisson
geometry are presented from a viewpoint adapted to orientable manifolds
according to the approach used in \cite{GMP}. Here, a Poisson tensor is
defined by a (m-2)-differential form satisfying the Jacobi identity
expressed in terms of forms and the trace operator calculus defined on the
orientable manifold M. In the section 4, we introduce the concept of
compatible 2-form with a regular foliation of even dimension and show that
the existence of a Poisson structure having as characteristic foliation the
given regular foliation is equivalent to the existence of a compatible
2-form. Moreover, we show that in such a case, the Poisson structure is
unimodular if and only if the first obstruction class vanishes.

Finally in section 5 for a given foliation, we show additional conditions
to the existence of a compatible 2-form, to have a transversally constant
Poisson structure in the sense of \cite{Vs}. We apply the given criteria
for the class of Dirac brackets associated to a set of independent smooth
functions on a symplectic manifold.

\section{Foliated cohomology and the first obstruction class.}
Let $M$ be an  $m$-dimensional smooth manifold. We consider a set $\{\alpha_1,\ldots, \alpha_k \}$ consisting of $k$ independent 1-forms which are globally defined on $M$ and satisfy the integrability conditions:
\begin{equation}
\di \alpha _{i}\wedge \mu =0;\ i=1,...k,  \label{1.1}
\end{equation}%
where $\mu $ is the k-form
\begin{equation}
\mu =\alpha _{1}\wedge \cdots \wedge \alpha _{k}.  \label{1.2}
\end{equation}
Let
\begin{equation*}
\mathscr{D}:= \{X\in \mathfrak{X}(M)  |
\alpha_i(X)=0 , \ i=1,2,\ldots,k  \} \}
\end{equation*}
the regular distribution induced by the 1-forms $\alpha
_{i},i=1,...,k$. Integrability condition (\ref{1.1}) implies that this distribution is integrable.
\\

Denote by $\mathcal{F}$ the codimension $k$ foliation of $M$ defined by the distribution $\mathscr{D}$ and refer to the set of $k$ independent 1-forms satisfying (\ref{1.1}) as \emph{generators of the foliation} $\mathcal{F}$.

Using a Riemannian metric one can have a transversal distribution to $%
\mathcal{F}$ and local dual vector fields $X^{j},$
\begin{equation}
\alpha _{i}(X^{j})=\delta _{ij},\ i,j=1,...,k  \label{1.2.1.1}
\end{equation}%
From equations (\ref{1.1}) and (\ref{1.2.1.1}), we have
\begin{equation}
\di \alpha _{i}=-\sum_{j=1}^{k}(L_{X^{j}}\alpha _{i})\wedge \alpha _{j}-\frac{1%
}{2}\sum_{j,r=1}^{k}(\mathrm{i}_{[X^{j},X^{r}]}\alpha _{i})\alpha _{j}\wedge \alpha
_{r},   \label{1.2.1}
\end{equation}
for $i=1,...,k$.
\bigskip

We can write the relations (\ref{1.2.1}) in matrix form
\begin{equation}
\left[
\begin{array}{c}
\di \alpha _{1} \\
\di \alpha _{2} \\
\vdots \\
\di \alpha _{k}%
\end{array}%
\right] =G\wedge \left[
\begin{array}{c}
\alpha _{1} \\
\alpha _{2} \\
\vdots \\
\alpha _{k}
\end{array}
\right]  \label{1.2.2}
\end{equation}
where $G=((G_{i}^{j}))$ is the matrix of 1-forms defined by
\begin{equation}
G_{i}^{j}=-((L_{X^{j}}\alpha _{i})+\frac{1}{2}%
\sum_{r=1}^{k}(\mathrm{i}_{[X^{r},X^{j}]}\alpha _{i})\wedge \alpha _{j}).
\label{1.2.3}
\end{equation}
Let be $\delta$ the 1-form given by the trace of matrix $G$
\begin{equation}
\delta =\operatorname{tr} G=-\sum_{i=1}^{k}L_{X^{i}}\alpha _{i}.  \label{1.2.7}
\end{equation}
If $\{\tilde{\alpha}_{1},\tilde{\alpha}_{2},...,\tilde{\alpha}_{k}\}$ is  another set of generators for $\mathcal{F}$  then there exists an invertible matrix of smooth function $F=((F_{i}^{j}))$ such that
\begin{equation}
\left[
\begin{array}{c}
\tilde{\alpha}_{1} \\
\tilde{\alpha}_{2} \\
\vdots \\
\tilde{\alpha}_{2}%
\end{array}%
\right] =F\left[
\begin{array}{c}
\alpha _{1} \\
\alpha _{2} \\
\vdots \\
\alpha _{k}
\end{array}
\right].  \label{1.2.4}
\end{equation}
Denote by $\tilde{G}$ the matrix of 1-forms (\ref{1.2.3}) corresponding to $\{\tilde{\alpha}_{1},\tilde{\alpha}_{2},...,\tilde{\alpha}_{k}\}$. By straightforward computation, we get
\begin{equation}
\tilde{G}=\di F\circ F^{-1}+F\circ G\circ F^{-1},  \label{1.2.5}
\end{equation}
and it follows that
\begin{equation}
\operatorname{tr}  \tilde{G}=\di (\ln \det F)+\operatorname{tr}  G  .\label{1.2.6}
\end{equation}
Therefore, the trace of the matrix $G$ (\ref{1.2.3}) defines a De Rham cohomology class independent of the choice of generators of $\mathcal{F}$. This class is an invariant associated to the foliation $\mathcal{F}$.

The expression in (\ref{1.2.1}) implies that  $\delta$ and $\mu$ are related by
\begin{equation}
\di \mu =-\delta \wedge \mu  \label{1.2.8},
\end{equation}
and
\begin{equation}
\di \delta \wedge \mu =0.  \label{1.2.8.1}
\end{equation}

For a regular foliation $\mathcal{F}$,
the following result generalize  Lemma 2 in \cite{GMP} for one-codimension foliations.

\begin{lemma}\label{lemma1}
A r-form $\gamma \in \Lambda ^{r}(M)$ vanishes on tangent vector fields to
the leaves of $\mathcal{F}$ if and only if $\gamma \wedge \mu =0.$
\end{lemma}
\begin{proof}
Consider the $k$-independent 1-forms $\alpha _{i},\ i=1,...,k$. Locally, there exists an independent set of 1-forms $\beta_{1},...,\beta _{m-k}$ such that $\{\alpha_1,\ldots,\alpha_k, \beta _{1},...,\beta _{m-k}\} $ is a local
basis for 1-forms. Consider a dual basis $X^{1},....,X^{k},Y^{1},...,Y^{m-k}$
with $\alpha _{i}(X^{j})=\delta _{ij},$ $\alpha _{i}(Y^{r})=0$, $\beta
_{s}(X^{j})=0$ and $\beta _{s}(Y^{r})=\delta _{sr}$ for $i,j=1,...,k$ and $%
r,s=1,...,m-k.$ Note that the vector fields $Y^{1},...,Y^{m-k}$ are tangent to
the leaves of $\mathcal{F}$. Then, $\gamma $ has in the given basis the
following representation
\begin{eqnarray*}
\gamma &=&\sum_{1\leq i_{1}\leq k}^{k}\theta _{1}^{i_{1}}\wedge \alpha
_{i_{1}}+\sum_{1\leq i_{1}<i_{2}\leq k}\theta _{2}^{i_{1}i_{2}}\wedge \alpha
_{i_{1}}\wedge \alpha _{i_{2}}+\cdots + \\
&&+\sum_{1\leq i_{1}<i_{2}<\cdots i_{r}\leq k}\theta _{r}^{i_{1}i_{2}\cdots
i_{r}}\wedge \alpha _{i_{1}}\wedge \alpha _{i_{2}}\wedge \cdots \wedge
\alpha _{i_{r}}+ \\
&&+\sum_{1\leq i_{1}<i_{2}<\cdots i_{r}\leq m-k}h_{r}^{i_{1}i_{2}\cdots
i_{r}}\beta _{i_{1}}\wedge \beta _{i_{2}}\wedge \cdots \wedge \beta _{i_{r}}
\end{eqnarray*}%
where $h_{r}^{i_{1}i_{2}\cdots i_{r}}\in C^{\infty }(M)$ and $\theta
_{j}^{i_{1}i_{2}\cdots i_{j}}\in \Lambda ^{r-j}(M)$ and $\mathrm{i}_{X^{p}}\theta
_{j}^{i_{1}i_{2}\cdots i_{j}}=0$ for all for $j=1,...r,$ and $p=1,...,k.$ If
$\gamma $ vanishes when valued on $r$ tangent vector fields, we have that all
functions $h_{r}^{i_{1}i_{2}\cdots i_{r}}$ vanishes and $\gamma \wedge \mu
=0.$ Note that, if $ m-k<r$, in the above expansion we have not
terms of the form $\beta _{i_{1}}\wedge \beta _{i_{2}}\wedge \cdots \wedge
\beta _{i_{r}}$ and also $\gamma \wedge \mu =0.$ Reciprocally, si $\gamma
\wedge \mu =0,$ evaluating $\gamma \wedge \mu $ on the dual vector fields $%
X^{i}$ of the 1-forms $\alpha_{i},$ one obtain for $\gamma $ the above
expansion without terms of the form $\beta _{i_{1}}\wedge \beta
_{i_{2}}\wedge \cdots \wedge \beta _{i_{r}}$ and consequently $\gamma $
vanishes when valued on tangent vector fields to the leaves of $\mathcal{F}$.
\end{proof}
\\

Lemma \ref{lemma1} allow us to define an equivalence relation on the exterior algebra $\Lambda (M)=\oplus _{i=0}^{m}\Lambda ^{i}(M)$ of
differential forms of $M$ as follows.
\begin{definition}\label{relequiv}
The r-forms $\beta ,\rho \in \Lambda ^{r}(M)$ are called  $\mathcal{F}$-equivalent if
\begin{equation}
(\beta -\rho )\wedge \mu =0.  \label{1.3}
\end{equation}
\end{definition}
From Lemma \ref{lemma1}, two forms $\beta $ and $\rho $ are $\mathcal{F}$-equivalent if
both forms  coincide on tangent vector fields to the foliation $\mathcal{F}$.

The relation of equivalence defined on $\Lambda (M)$ by (\ref{1.3}) does not depend on the choice of generators of the foliation $\mathcal{F}$. In fact, if $\tilde{\alpha}_{1},\tilde{\alpha}_{2,}...\tilde{\alpha}_{k}$ is
another set of generators of $\mathcal{F}$, then relation  (\ref{1.2.4}) implies that
\begin{equation}
\tilde{\alpha}_{i}=\sum_{j=1}^{k}F_{i}^{j}\alpha _{j}, \ i,j=1,...,k.
\label{1.4}
\end{equation}
It follows form here that $\tilde{\mu}=\tilde{\alpha}_{1}\wedge \tilde{\alpha}_{2}\wedge \cdots
\wedge \tilde{\alpha}_{k}=(\det F)\mu $. Therefore, $(\beta -\rho )\wedge \mu =0$
if and only if $(\beta -\rho )\wedge \tilde{\mu}=0.$

For every $r=1,2,...m-k$, we denote by $\Lambda _{\mathcal{F}}^{r}(M)$ the space of
classes of equivalence and $\pi $ the projection operator sending each
r-form $\beta $ into its equivalence class
\begin{eqnarray}
\pi &:&\Lambda ^{r}(M)\rightarrow \Lambda _{\mathcal{F}}^{r}(M)  \label{1.5}
\\
\pi (\beta ) &=&\left\{ \rho \in \Lambda ^{r}(M)\text{ with }(\beta -\rho
)\wedge \mu =0\right\}  \notag
\end{eqnarray}%
Note that $\Lambda _{\mathcal{F}}^{r}(M)=0$ for all $r>m-k.$ From now on, the
classes $\pi (\beta )\in \Lambda _{\mathcal{F}}^{r}(M)$ will be called
\textit{foliated r-forms on }$M.$ The foliated r-forms take the general form%
\begin{eqnarray}
\pi (\beta ) &=&\beta +\sum_{i=1}^{k}\theta _{1}^{i}\wedge \alpha
_{i}+\sum_{1\leq i_{1}<i_{2}\leq k}\theta _{2}^{i_{1}i_{2}}\wedge \alpha
_{i_{1}}\wedge \alpha _{i_{2}}+\cdots +  \label{1.5.1} \\
&&+\sum_{1\leq i_{1}<i_{2}<\cdots i_{r}\leq k}\theta
_{r-1}^{i_{1}i_{2}\cdots i_{r}}\wedge \alpha _{i_{1}}\wedge \alpha
_{i_{2}}\wedge \cdots \wedge \alpha _{i_{r}}  \notag
\end{eqnarray}%
for some $\theta _{j}^{i_{1}i_{2}\cdots i_{j}}\in \Lambda ^{r-j}(M)$ for $j=1,...r.$
\\

It is easy to prove that if $\alpha \in \pi(\beta)$ and $\gamma \in \pi(\chi)$ then $\alpha + \gamma \in \pi(\beta + \chi)$ and $\alpha \wedge \gamma \in \pi(\beta \wedge \chi).$ Taking into account this property, we can extend the notion of  sum and wedge product to the space of foliated forms $ \Lambda _{\mathcal{F}}^{r}(M)$. Given $\beta,\eta\in \Lambda ^{r}(M)$  and $\chi \in \Lambda^{s}(M)$ we define the sum of $\pi(\beta) $ and $ \pi( \eta)$ by
\begin{equation}
\pi(\beta) + \pi(\eta) := \pi(\beta + \eta).\label{sumafol}
\end{equation}
The wedge product of $\pi(\beta) $ and $ \pi( \chi)$ is defined by
\begin{equation}
\pi (\beta )\wedge \pi (\chi ):=\pi (\beta \wedge \chi ).  \label{1.7}
\end{equation}

By straightforward computation, we can prove that the sum and wedge product operation on $\Lambda _{\mathcal{F}}(M)$ have the same properties of the usual operations on  $\Lambda (M)$.

\begin{lemma}\label{lemderiv}
Let $\beta, \eta \in \Lambda^r(M)$. If $\beta$ and $\eta$ are $\mathcal{F}$-related then so they are $\di \beta$ and $\di \eta $. In particular, if $\beta \wedge \mu =0$ then $\di\beta \wedge \mu=0 .$
\end{lemma}
\begin{proof}
$\beta$ and $\eta$ are $\mathcal{F}$-related if and only if
\begin{equation*}
(\beta - \eta)\wedge \mu =0.
\end{equation*}
Thus, we get
\begin{equation*}
(\di \beta - \di \eta)\wedge\mu + (-1)^k (\beta - \eta)\wedge \di \mu =0.
\end{equation*}
From (\ref{1.2.8}), we have $(\beta - \eta)\wedge \di \mu = -(-1)^{k}(\beta - \eta)\wedge \mu \wedge \delta =0$. Therefore
\begin{equation*}
(\di \beta - \di \eta)\wedge\mu =0,
\end{equation*}
and $\di \beta$ and $\di \eta $ are $\mathcal{F}$-related. Taking into account that $\beta \wedge \mu =0$ means $\beta$ is $\mathcal{F}$-related with $\eta \equiv 0$, we have the rest of the lemma.
\end{proof}
\begin{remark}
Notice that if $\theta \in \Lambda ^{s}(M)$ and $\mu \wedge \theta =0$ then  $\mu \wedge \mathrm{i}_{X}\theta
=0 $ for all $X$ such that $\mathrm{i}_{X}\mu =0.$
\end{remark}
%From Lemma \ref{lemderiv}, we have that $\pi (\beta )=\pi (\rho )$ implies $\pi (d\beta)=\pi (d\rho ).$

Taking into account Lemma \ref{lemderiv}, we introduce the foliated exterior derivative operator $\di_{ \mathcal{F}}$ as follows:
\begin{eqnarray}
\di_{\mathcal{F}} &:&\Lambda _{\mathcal{F}}^{r}(M)\rightarrow \Lambda _{\mathcal{F}}^{r+1}(M)  \label{1.8} \\
\di_{\mathcal{F}}(\pi (\beta )) &:=&\pi (\di\beta ).  \notag
\end{eqnarray}%

\begin{proposition}\label{profolder}
The foliated exterior derivative $\di_{ \mathcal{F}}$ has the following properties:
\begin{itemize}
\item[(i)] $\di_{ \mathcal{F}}$ is a linear operator. That is, for $\beta, \eta \in \Lambda^r(M)$,
\begin{equation*}
\di_\mathcal{F} (\pi(\beta+\eta))= \di_\mathcal{F} (\pi(\beta))+\di_\mathcal{F} (\pi(\eta)).
\end{equation*}
\item[(ii)] $\di_{ \mathcal{F}}$ is a derivation with respect to the foliated wedge product. For $\beta \in \Lambda^r(M),\chi \in \Lambda^r(M)$, we have
    \begin{equation*}
    \di_\mathcal{F}( \pi(\beta\wedge\chi))= \di_\mathcal{F} (\pi(\beta))\wedge \pi(\chi)+(-1)^r\pi(\beta)\wedge \di_\mathcal{F} (\pi(\chi)).
    \end{equation*}
\item[(iii)] $\di_{ \mathcal{F}}$ is a cohomology operator,
\begin{equation}
\di_{\mathcal{F}}\circ \di_{\mathcal{F}}=0.  \label{1.9}
\end{equation}%
%here, $0$ stands for $\pi(0).$
\end{itemize}
 \end{proposition}
The proof of the proposition above consists of straightforward computations.
%\begin{proof}
%This proof consist of straightforward computations.
%\begin{itemize}
%\item[(i)]Let $\beta, \eta \in \Lambda^r(M)$. By equation (\ref{sumafol}), get
%\begin{equation*}
%\di_\mathcal{F} (\pi(\beta+\eta))= \pi(\di(\beta+\eta))=\pi(\di\beta+\di\eta)=\pi(\di\beta)+\pi(\di\eta) = \di_\mathcal{F} (\pi(\beta))+\di_\mathcal{F} (\pi(\eta)).
%\end{equation*}
%\item[(ii)] Now, we take $\beta \in \Lambda^r(M),\chi \in \Lambda^r(M)$. Equation (\ref{sumafol}) (\ref{1.7}), we have

%\begin{eqnarray*}
%    \di_\mathcal{F}( \pi(\beta\wedge\chi))&=&\pi(\di (\beta\wedge\chi))=\pi(\di \beta\wedge\chi +(-1)^r \beta\wedge\di\chi), \\ &=& \pi(\di \beta)\wedge\pi(\chi) +(-1)^r \pi( \beta)\wedge\pi(\di\chi),\\
%    &=&\di_\mathcal{F} (\pi(\beta))\wedge \pi(\chi)+(-1)^r\pi(\beta)\wedge \di_\mathcal{F} (\pi(\chi)).
%    \end{eqnarray*}
%\item[(iii)] Finally, if $\beta \in \Lambda^r(M)$
%\begin{equation*}
%\di_{\mathcal{F}}\circ \di_{\mathcal{F}}(\pi(\beta)) = \di_{\mathcal{F}}(\pi(\di\beta))=\pi(\di\circ\di(\beta))=\pi(0).
%\end{equation*}
%\end{itemize}
%\end{proof}

For $r=0,...,m-k$, we define the foliated cohomology spaces
\begin{equation}
\mathcal{H}_{\mathcal{F}}^{r}(M)=\frac{\ker \di_{\mathcal{F}}:\Lambda
_{\mathcal{F}}^{r}(M)\rightarrow \Lambda _{\mathcal{F}}^{r+1}(M)}{\operatorname{Im}%
\di_{\mathcal{F}}:\Lambda _{\mathcal{F}}^{r-1}(M)\rightarrow \Lambda _{\mathcal{F}}^{r}(M)}.  \label{1.10}
\end{equation}%
A foliated r-form $\pi (\beta )$ is called $\di_{\mathcal{F}}-closed$ if $\di_{\mathcal{F}}(\pi (\beta ))=0.$ $\pi (\beta )$ is called $\di_{\mathcal{F}}-exact$ if there exists $\phi \in \Lambda ^{r-1}(M)$ such that $\pi (\beta )=\di_{\mathcal{F}}(\pi
(\phi )).$ For each $\di_{\mathcal{F}}-closed$ foliated form $\pi (\beta )$ we
denote by $[\pi (\beta )]_{\mathcal{F}}$ its foliated cohomology class. If a
closed foliated r-form $\pi (\beta )$ is $\di_{\mathcal{F}}-exact,$ then its
cohomology class $[\pi (\beta )]_{\mathcal{F}}$ vanishes.

The definition of foliated form and foliated cohomology for regular
foliations given here, coincide with the one given in \cite{HMS}.

Let $\mathcal{F}$ a regular foliation on $M$ and $\{\alpha_1, \ldots, \alpha_k \}$ a generator set of $\mathcal{F}$.
Consider the 1-form $\delta $ given by (\ref{1.2.7}).
\begin{lemma}\label{lemobs}
The 1-form $\delta $ (\ref{1.2.7}) has the following properties:
\begin{itemize}
\item[(i)] If $\{\tilde{\alpha}_{1},\tilde{\alpha}_{2,}...\tilde{\alpha}_{k}\}$ is another generator set of
 $\mathcal{F}$, $\tilde{X}^j$ the dual vector fields  and $\tilde{\delta}=-\sum_{i=1}^{k}L_{\tilde{X}^{i}}\tilde{\alpha}
_{i}$  then
\begin{equation*}
\di\tilde{\delta}=\di\delta. \label{1.17}
\end{equation*}

\item[(ii)]The foliated 1-form $\pi(\delta )$ does not depend on the dual vector fields used in the definition
of $\delta $ (\ref{1.2.7}).
\item[(iii)] $\pi (\delta )$ is $\di_{\mathcal{F}}$-closed .

\item[(iv)] The foliated cohomology class $[\pi (\delta )]_{\mathcal{F}}$ of $\mathcal{H}_{\mathcal{F}}^{1}(M)$ does not depend of the generators of $\mathcal{F}$.
\end{itemize}
\end{lemma}
\begin{proof}
\begin{itemize}
\item[(i)] From (\ref{1.2.6}) $\delta - \tilde{\delta}$ is an exact 1-form. Thus, $\di\delta -\di \tilde{\delta}=0.$
\item[(ii)] If $\tilde{X}^{j},$ $j=1,...,k$ is another set of global dual vector fields to $\alpha _{i},$ $i=1,...,k,$ then $\tilde{X}^{j}=X^{j}+T^{j}$ where $T^{j}$ are vector fields tangent to $\mathcal{F}$. In this case, we have
\begin{equation*}
\tilde{\delta}= -\sum_{i=1}^{k}L_{\tilde{X}^{i}}\alpha
_{i} =\delta-\sum_{i=1}^{k}L_{T^{i}}\alpha _{i}.  \label{1.14.1}
\end{equation*}
Since $(\sum_{i=1}^{k}L_{T^{i}}\alpha _{i})\wedge \mu =0$, we obtain $\pi(\tilde{\delta})=\pi(\delta )$.
\item[(iii)] From (\ref{1.2.8}) $\di\delta \wedge \mu =0$. Hence,
\begin{equation*}
\di_{\mathcal{F}} (\pi(\delta))=\pi(\di\delta) = \pi(0).
\end{equation*}

\item[(iv)] Let $\tilde{\delta}$ as in item (i). By (\ref{1.2.6}), we have $\tilde{\delta}= \delta + \di (\ln \det F)$. Thus,
    \begin{equation*}
    \pi(\tilde{\delta})= \pi(\delta)+ \pi(\di (\ln \det F))=\pi(\delta)+\di_{\mathcal{F}}( \pi( \ln \det F)).
    \end{equation*}
Therefore, $[\pi (\delta )]_{\mathcal{F}}=[\pi (\tilde{\delta} )]_{\mathcal{F}}$.
\end{itemize}
\end{proof}

Following the work of Guillemin, Miranda and Pires in \cite{GMP}, we call to the foliated cohomology class
\begin{equation}
C_{\mathcal{F}}=[\pi (\delta )]_{\mathcal{F}}  \label{1.18}
\end{equation}%
\textit{the first obstruction class} of $\mathcal{H}_{\mathcal{F}}^{1}(M)$. By Lemma (\ref{lemobs}), this class is an invariant of the regular foliation $\mathcal{F}$.

The following theorem is the generalization of Theorem 4 given in \cite{GMP} for codimension one regular foliations.

\begin{theorem}
For a regular foliation $\mathcal{F}$ generated by k independent 1-forms, the first obstruction class $C_{\mathcal{F}}$ vanishes identically if and only if, we can chose the defining 1-forms $\alpha _{i}$ $i=1,...,k$ such that $\mu =\alpha_{1}\wedge \alpha_{2}\wedge \cdots \wedge \alpha _{k}$ to be closed.
\end{theorem}
\begin{proof}
Suppose that $C_{\mathcal{F}}$ vanishes. Then, there exists a smooth function $h\in C^{\infty }(M)$ such that  $\delta \wedge \mu =\di h\wedge \mu$. By (\ref{1.2.8}), $\di \mu =- \di h\wedge \mu .$ Now, take the following generators $\tilde{\alpha}_{1}=e^{-h}\alpha _{1},\tilde{\alpha}_{2}=\alpha _{2}...\tilde{\alpha}_{k}=\alpha _{k},$ and its
corresponding dual vector fields $\tilde{X}_{1}=e^{h}X_{1},$ $\tilde{X}%
_{2}=X_{2},...,\tilde{X}_{k}=X_{k}.$ Thus, $\tilde{\delta}=-L_{\tilde{X}_{1}}%
\tilde{\alpha}_{1}-\sum_{j=2}^{k}L_{X^{j}}\alpha _{i}=-(L_{X_{1}}h)\alpha
_{1}-\di h+\delta $. Finally, $\di\tilde{\mu}=- \tilde{\delta} \wedge \tilde{\mu}=e^{-h}((L_{X_{1}}h)\alpha
_{1}+\di h-\delta )\wedge \mu = 0.$ Reciprocally, if $\di \mu =0,$ then $\delta \wedge \mu =0$ and $[\pi (\delta )]=0$.
\end{proof}

\section{Poisson structures on orientable manifolds.}

In this section, we recall some basic facts on Poisson structures and
extending the ideas given in \cite{GMP}, we present an adapted approach
for Poisson structures on orientable manifolds based on the trace operator
calculus on the graduated algebra of contravariant antisymmetric tensor
fields on $M$.

Let be $M$ an orientable $m-$manifold and $\Omega $ a volume form on $M.$
For $k=0,1,...m,$ denote by $\mathcal{V}^{k}(M)$ the space of contravariant
antisymmetric k-tensors on $M.$ The trace operator $D_{\Omega }:\mathcal{V}%
^{k}(M)\rightarrow \mathcal{V}^{k-1}(M)$ is defined on $A\in \mathcal{V}%
^{k}(M)$ through the expression%
\begin{equation}
i_{D_{\Omega }(A)}\Omega =di_{A}\Omega  \label{3.0}
\end{equation}%
In particular, for a vector field $X$ $\in \mathcal{V}^{1}(M)$ we have
\begin{equation}
i_{D_{\Omega }(X)}\Omega =L_{X}\Omega =D_{\Omega }(X)\Omega  \label{3.0.1}
\end{equation}%
and the smooth function $D_{\Omega }(X)$ is called the divergence of $X$
with respect the volume form $\Omega $ and denoted by $\operatorname{div}_{\Omega
}X. $ The trace operator $D_{\Omega }$ is a cohomology operator, ie. $%
D_{\Omega }\circ D_{\Omega }=0$ and on tensor fields $A\in \mathcal{V}^{r}(M)
$ and $B\in \mathcal{V}^{k}(M)$ it has the property
\begin{equation}
D_{\Omega }(A\wedge B)=(-1)^{k}D_{\Omega }(A)\wedge B+A\wedge D_{\Omega
}B+(-1)^{k+r+1}[A,B]  \label{3.0.2}
\end{equation}%
Here, $[,]$ denotes the Schouten bracket operation defined on the graduated
algebra of contravariant antisymmetric tensor fields on $M.$ For vector
fields $X,Y\in \mathcal{V}^{1}(M)$ the above formula reads%
\begin{equation}
D_{\Omega }(X\wedge Y)=-\operatorname{div}_{\Omega }(X)\wedge Y+X\wedge \operatorname{div}%
_{\Omega }Y-[X,Y]  \label{3.0,3}
\end{equation}%
and for $A\in \mathcal{V}^{r}(M)$ and $B\in \mathcal{V}^{k}(M)$%
\begin{equation}
D_{\Omega }[A,B]=[D_{\Omega }A,B]+[A,D_{\Omega }B]  \label{3.0.4}
\end{equation}

For more complete information on the Schouten bracket and the trace operator
see \cite{DF,K,Vs}.

A Poisson structure on a $m-$manifold $M$ is given by a contravariant
antysimmetric 2-tensor $\Pi $ satisfying the Jacobi Identity $[\Pi ,\Pi ]=0.$
A manifold $M$ equipped with a Poisson structure $\Pi $ is called a Poisson
manifold.

We denote by $\Pi ^{\#}$ the morphism
\begin{eqnarray}
\Pi ^{\#} &:&T^{\ast }M\rightarrow TM  \label{2.1} \\
\beta (\Pi ^{\#}(\eta )) &=&\Pi (\eta ,\beta ),\text{ }\forall \eta ,\beta
\in T^{\ast }(M)  \label{2.1.1}
\end{eqnarray}%
At each point $p\in M,$ the $rank$ of $\Pi $ at $p$ is the dimension of the
linear space $\Pi ^{\#}(T_{p}^{\ast }M).$ The Poisson tensor is called
regular if its rank is constant on $M.$ For each smooth function $f$ on $M,$
the vector field $X_{f}=\Pi ^{\#}(df)$ is called a Hamiltonian vector field
and function $f$ its Hamiltonian function.

For constant rank Poisson tensors $\Pi ,$ the set of Hamiltonian vector
fields generates an integrable distribution in the sense of Frobenious and
consequently a regular foliation $\mathcal{F}$ whose leaves are even
dimensional submanifolds. This foliation is called the Characteristic
foliation of $\Pi $ and each of its leaves $L$ carries a symplectic
structure with symplectic form $\omega _{L}$ defined on the restriction of
Hamiltonian vector fields to $L$ by%
\begin{equation*}
\omega _{L}(X_{f},X_{g})=\Pi (df,dg)
\end{equation*}

If the Characteristic foliation $\mathcal{F}$ is defined by $k=m-2r$
independent 1-forms $\alpha _{i},$ $i=1,...,k$ satisfying the integrability
conditions (\ref{1.1}), we have $\Pi ^{\#}(\alpha _{i})=0$ for $i=1,...,k$
and $\alpha _{i}(\Pi ^{\#}(\beta ))=0$ for all $\beta \in \Lambda ^{1}(M).$
The dimension of leaves is $2r=m-k.$

On an orientable $m-$manifold $M$ with a volume form $\Omega ,$ any 2-tensor
$\Pi $ is defined by a unique differential form $\sigma \in \Lambda
^{m-2}(M) $ through the formula%
\begin{equation}
i_{\Pi }\Omega =\sigma  \label{2.1.2}
\end{equation}%
From (\ref{3.0.2}) and for $f\in C^{\infty }(M)$ and $X\in \mathcal{%
V}^{1}(M),$ we have
\begin{eqnarray}
i_{X_{f}}\Omega &=&-df\wedge \sigma  \label{2.1.3} \\
L_{X}\sigma &=&(\operatorname{div}_{\Omega }X)\sigma +i_{[\Pi ,X]}\Omega
\label{2.1.4}
\end{eqnarray}%
and the Jacobi Identity for $\Pi $ takes the form
\begin{equation}
L_{X_{f}}\sigma =\operatorname{div}_{\Omega }(X_{f})\text{ }\sigma ,\text{ }\forall
f\in C^{\infty }(M).  \label{2.1.5}
\end{equation}

The infinitesimal automorphisms of the Poisson tensor $\Pi $ or Poisson
vector fields, are those vector fields $X$ on $M$ satisfying $[\Pi ,X]=0.$
In terms of the form $\sigma $, this last conditions reads
\begin{equation}
L_{X}\sigma =(\operatorname{div}_{\Omega }X)\sigma  \label{2.1.6}
\end{equation}

In particular, the vector field $Z_{\Omega }=D_{\Omega }(\Pi )$ is a Poisson
vector field with $\operatorname{div}_{\Omega }(Z_{\Omega })=0.$ The vector field $%
Z_{\Omega }$ is called the modular vector field associated to the volume
form $\Omega $ and it has been introduced by A. Weinstein in \cite{W}. The
modular vector fields controls the divergence of Hamiltonian vector field $%
X_{f}$
\begin{equation}
\operatorname{div}_{\Omega }X_{f}=L_{Z_{\Omega }}f.  \label{2.5.3}
\end{equation}%
In fact, if $\tilde{\Omega}=h\Omega $ with $h\neq 0$ is another volume form,
we have for the corresponding modular vector field, the relation%
\begin{equation}
Z_{\tilde{\Omega}}=Z_{\Omega }-X_{\ln (h)}  \label{2.5.2}
\end{equation}%
and one notice that the modular vector fields $Z_{\Omega }$ associated to
the volume forms $\Omega $, belong to the same class of the Poisson
1-cohomology, \cite{DF}. This class is called the unimodular class of $\Pi $. If for
some volume form $\Omega $, the vector field $Z_{\Omega }$ is a Hamiltonian
vector field $Z_{\Omega }=X_{g}$ for some smooth function $g$, from (\ref%
{2.5.2}) we have that $Z_{\tilde{\Omega}}$ vanishes for the volume form $\tilde{\Omega}%
=e^{g}\Omega $ and all Hamiltonian vector fields preserve the volume form $%
\tilde{\Omega}$. When the modular vector field is a Hamiltonian vector field
the Poisson structure is called an unimodular Poisson structure \cite{W}.
For more complete information on Poisson structures, see \cite{DF,LPV,Vs}.

If the Characterist foliation of the Poisson tensor $\Pi $ is the regular
foliation $\mathcal{F}$ defined by k-independent 1-forms $\alpha
_{i},i=1,...,k,$ then the (m-2)- form $\sigma $ in (\ref{2.1.2}) takes the
expression%
\begin{equation*}
\sigma =\mu \wedge \theta \text{ }
\end{equation*}%
with $\mu $ as in (\ref{1.2}) and $\theta $ some ($m-k-2)-$form.

Let us consider now, the 1-form $\delta $ given in (\ref{1.2.7}) and
belonging to the first obstruction class of $\mathcal{F}$. From (\ref%
{1.2.8.1}), the 2-form $d\delta $ vanishes when valued on vector fields
tangent to the leaves of $\mathcal{F}$ and by a well known result \cite{Vs}, the vector field $\Pi ^{\#}(\delta )$ is a Poisson vector field tangent
to the leaves of $\mathcal{F}$. If one takes another set of 1-forms
generating $\mathcal{F}$ and consider the corresponding 1-form $\tilde{\delta%
}$, from (\ref{1.2.6}), we obtain%
\begin{equation}
\Pi ^{\#}(\delta )-\Pi ^{\#}(\tilde{\delta})=\Pi (d(\ln \det F))
\label{2.5.4}
\end{equation}%
and the Poisson vector field $\Pi ^{\#}(\delta )$ remains in the same class
of the Poisson 1-cohomology of $\Pi $.

In general, if $M$ is an oriented $m-$manifold with $\Omega $ a volume form
and $\mathcal{F}$ a regular foliation generated by k-independent 1-forms $%
\alpha _{i},$ $i=1,...,k,$ for any 2-tensor $\Pi $ $($not necessarly a
Poisson tensor) we have%
\begin{equation*}
i_{\Pi }\Omega =\mu \wedge \theta
\end{equation*}%
for some $\theta \in \Lambda ^{m-k-2}(M).$ Moreover, for the vector fields $%
\Pi ^{\#}(\delta )$ and $D_{\Omega }(\Pi )$ we have
\begin{eqnarray}
i_{\Pi ^{\#}(\delta )}\Omega &=&-\delta \wedge \mu \wedge \theta =d\mu
\wedge \theta =  \label{2.5.5} \\
&=&d(\mu \wedge \theta )-(-1)^{k}\mu \wedge d\theta =  \notag \\
&=&i_{D_{\Omega }(\Pi )}\Omega -(-1)^{k}\mu \wedge d\theta  \notag
\end{eqnarray}

Then, the closed foliated (m-1)-forms $i_{\Pi ^{\#}(\delta )}\Omega $ and $%
i_{D_{\Omega }(\Pi )}\Omega $ differ by an exact foliated form and
consequently they belong to the same foliated cohomology class. $\ \ $

\section{Compatible 2-forms with regular foliations on orientable
manifolds.}

Let be $\mathcal{F}$ a regular foliation on an orientabe manifold $M,$
generated by $k$ global independent 1-forms $\alpha _{i},$ $i=1,...,k$ with $%
k=m-2r$ satisfying the integrability condition\ (\ref{1.1}). Let $\omega$ be a
2-form such that $\Omega _{\omega }=\mu \wedge\omega ^{r}$ is a volume form on $M.$ Denote by $\Pi $ the 2-tensor defined
by the relation
\begin{equation}
\mathrm{i}_{\Pi }\Omega _{\omega }=r\mu \wedge \omega ^{r-1}  \label{4.1}
\end{equation}
For any $\beta \in \Lambda ^{1}(M),$ we have $i_{\Pi ^{\#}(\beta )}\mu =0.$ That is, $\Pi ^{\#}(\beta)$ is a vector field tangent to the foliation $\mathcal{F}$. In particular, for each smooth function $f$, $X_{f}=\Pi ^{\#}(\di f)$ satisfies the expression
\begin{equation}
\mathrm{i}_{X_{f}}\Omega _{\omega }=-r\di f\wedge \mu \wedge
\omega ^{r-1}.\label{ecsigma}
\end{equation}

\begin{lemma}\label{lemcompa}
If $\omega $ is a 2-form such that $\Omega _{\omega }=\mu \wedge \omega ^{r}$
is a volume form and $\Pi $ is the 2-tensor such that $\mathrm{i}_{\Pi }\Omega _{\omega
}=r\mu \wedge \omega ^{r-1},$ then for each smooth function $f,$ we have,
\begin{equation}
(\mathrm{i}_{X_{f}}\omega +\di f)\wedge \mu =0  \label{3.3.1}
\end{equation}
and
\begin{equation}
\omega (X_{f},X_{g})=\Pi (\di f,\di g),\text{ }\forall f,g\in \mathcal{C}^{\infty
}(M)  \label{3.3.2}
\end{equation}
\end{lemma}
\begin{proof}
From (\ref{ecsigma}), we have
\begin{eqnarray*}
\mathrm{i}_{X_f}\Omega _{\omega }&=& r \mu \wedge \mathrm{i}_{X_f}\omega \wedge \omega^{r-1}.
\end{eqnarray*}
Therefore, we have $(\mathrm{i}_{X_{f}}\omega
+df)\wedge \mu \wedge \omega ^{r-1}=0.$ Suppose that $(\mathrm{i}_{X_{f}}\omega
+df)\wedge \mu \neq 0.$ Then, $(\mathrm{i}_{X_{f}}\omega +df),$ $\alpha _{i},$  are $k+1$ independent 1-forms and we can complete to a local basis $\mathcal{B}$ of 1-forms. Locally, $\omega =(\mathrm{i}_{X_{f}}\omega +df)\wedge \rho +\chi $ where $\chi \in
\Lambda ^{2}(M)$ and $\chi $ does not contain\ $(\mathrm{i}_{X_{f}}\omega +df)$ in
its decomposition. Then $\omega ^{r-1}=\chi ^{r-1}+(r-1)\chi \wedge
(\mathrm{i}_{X_{f}}\omega +df)\wedge \rho $. Taking the wedge product with $\mu
\wedge (\mathrm{i}_{X_{f}}\omega +df)$, we obtain $0=\chi ^{r-1}\wedge \mu \wedge
(\mathrm{i}_{X_{f}}\omega +df)$. From the property of $\chi $, we get $0=\chi
^{r-1}\wedge \mu .$ Finally, $\omega ^{r}=\chi ^{r}+r\chi \wedge
(\mathrm{i}_{X_{f}}\omega +df)\wedge \rho $ which implies that $\omega ^{r}\wedge \mu =0.$ But this is a contradiction because of $\Omega _{\omega }=\omega ^{r}\wedge
\mu $ is a volume form. Then $(\mathrm{i}_{X_{f}}\omega +df)\wedge \mu =0.$
For the second part of Lemma, we only compute
\begin{eqnarray*}
0=\mathrm{i}_{X_g}(\mathrm{i}_{X_{f}}\omega +\di f)\wedge \mu = (\omega(X_f,X_g)+\Pi(\di g , \di f))\mu.
\end{eqnarray*}
Since $\mu \neq 0,$ we have $\omega(X_f,X_g)+\Pi(\di g , \di f)=0. $
\end{proof}

Using the previous geometrical objects associated to the regular foliation $\mathcal{F}$, we give the following definition.

\begin{definition}
A 2-form $\omega \in \Lambda ^{2}(M)$ is compatible with the regular
foliation $\mathcal{F}$ if: a) $\di\omega \wedge \mu =0,$ and b) $\Omega
_{\omega }=\mu \wedge \omega ^{r}$ is a volume form on $M$.
\end{definition}

Notice that condition a) is equivalent to state that $\omega $ is a closed
2-form when valued on tangent vector fields to $\mathcal{F}$ and b) implies $%
\omega $ is nondegenerate on the leaves of $\mathcal{F}$. If $%
\omega $ is a compatible 2-form with $\mathcal{F}$ and we have dual vector
fields $X^{r}$ with $\alpha _{j}(X^{r})=\delta _{rj}$ for $i,j=1,...,k.$,
then, the 2-form
\begin{equation*}
\tilde{\omega}=\omega +\sum_{i=1}^{k}(\mathrm{i}_{X^{i}}\omega )\wedge \alpha _{i}+%
\frac{1}{2}\sum_{i,j=1}^{k}(\mathrm{i}_{X^{i}\wedge X^{j}}\omega )\alpha _{i}\wedge
\alpha _{j}
\end{equation*}
is a new compatible 2-form satisfying
\begin{equation*}
\mathrm{i}_{X^{r}}\tilde{\omega}=0,
\end{equation*}%
for $r=1,...,k.$

The existence of a compatible 2-form $\omega $ with a regular foliation $%
\mathcal{F}$ implies the existence of a Poisson tensor having $\mathcal{F}$
as its characteristic foliation as we show in the following result.

\begin{theorem}
Let be $M$ an orientable manifold and $\mathcal{F}$ a regular foliation
generated by k independent $1-$forms $\alpha _{1},\alpha _{2},...,\alpha
_{k}.$ Then, there exists a Poisson tensor $\Pi $ having $\mathcal{F}$ as
its characteristic foliation if and only if there exists a 2-form $\omega $
compatible with $\mathcal{F}$. Moreover, $\Pi $ is a unimodular Poisson
tensor if and only if the first obstruction class $C_{\mathcal{F}}$ vanishes.
\end{theorem}
\begin{proof}
First, we assume that the 2-form $\omega $ is compatible with $\mathcal{F}$. We consider the 2-tensor $\Pi $ defined by the
relation $\mathrm{i}_{\Pi }\Omega _{\omega }=r\mu \wedge \omega ^{r-1}.$ We only need to prove that $\Pi$ satisfies the
the Jacobi Identity. It is sufficient to show that given a smooth function $f$, the vector field $X_f= \Pi^\#(\di f)$
satisfies $L_{X_{f}}(\mu \wedge \omega ^{r-1})=\operatorname{div}_{\Omega _{\omega
}}(X_{f})\mu \wedge \omega ^{r-1}$.

Based on the properties (\ref{1.2.8}) and (%
\ref{3.3.1}) of $\mu $ and $\omega ,$ we have
\begin{equation*}
L_{X_{f}}(r\mu \wedge \omega ^{r-1})=rL_{X_{f}}\mu \wedge \omega
^{r-1}+r(r-1)\mu \wedge L_{X_{f}}\omega \wedge \omega ^{r-2}
\end{equation*}%
On other hand, we get%
\begin{equation*}
L_{X_{f}}\mu =\mathrm{i}_{X_{f}}d\mu =-(\mathrm{i}_{X_{f}}\delta )\mu
\end{equation*}%
and%
\begin{equation*}
\mu \wedge L_{X_{f}}\omega =\mu \wedge (di_{X_{f}}\omega +\mathrm{i}_{X_{f}}d\omega ).
\end{equation*}
Computing the exterior derivative of (\ref{3.3.1}), we have
\begin{equation*}
\mu \wedge (di_{X_{f}}\omega )=(-1)^{k+1}\delta \wedge \mu \wedge
(df+\mathrm{i}_{X_{f}}\omega )=0.
\end{equation*}%
Since $\mathrm{i}_{X_f}(\mu \wedge \di \omega)=0,$ we obtain
\begin{equation*}
\mu \wedge \mathrm{i}_{X_{f}}d\omega )=(-1)^{k}\mathrm{i}_{X_{f}}(\mu \wedge d\omega )=0.
\end{equation*}%
Combining the equations above
\begin{equation*}
L_{X_{f}}(\mu \wedge \omega ^{r-1})=-(\mathrm{i}_{X_{f}}\delta )\mu \wedge \omega
^{r-1}.
\end{equation*}%
Finally,
\begin{eqnarray*}
\operatorname{div}_{\Omega _{\omega }}(X_{f})\mu \wedge \omega ^{r} &=&L_{X_{f}}(\mu
\wedge \omega ^{r})=L_{X_{f}}(\mu \wedge \omega ^{r-1})\wedge \omega +\mu
\wedge \omega ^{r-1}\wedge L_{X_{f}}\omega = \\
&=&(-\mathrm{i}_{X_{f}}\delta )\mu \wedge \omega ^{r}.
\end{eqnarray*}%
Therefore,
\begin{equation*}
rL_{X_{f}}(\mu \wedge \omega ^{r-1})=\operatorname{div}_{\Omega _{\omega
}}(X_{f})r\mu \wedge \omega ^{r-1}.
\end{equation*}%
 If $C_{\mathcal{F}}$ vanishes, $\delta
\wedge \mu =dh\wedge \mu $ for some smooth function $h$ and $\mathrm{i}_{X_{f}}\delta
=L_{X_{f}}h.$ So, $L_{D_{\Omega _{\omega }}(\Pi )}f=L_{X_{h}}f$ $\ \forall
f\in C^{\infty }(M)$ and $D_{\Omega _{\omega }}(\Pi )=X_{h}$ is a
Hamiltonian vector field\ and $\Pi $ a unimodular Poisson tensor. Note that $%
\operatorname{div}_{\Omega _{\omega }}(X_{f})=-\mathrm{i}_{X_{f}}\delta .$ Reciprocally,
suppose that $\Pi $ is a unimodular Poisson tensor. %Then the 2-form $\omega $ defined on Hamiltonian vector fields by%
Then $D_\Omega(\Pi)= X_h $ for some $h$ and $dh\wedge \mu \wedge \omega^{r-1} = d\mu \wedge \omega^{r-1}=\delta\wedge\mu\wedge\wedge \omega^{r-1} $. Finally, $(dh - \delta)\wedge\mu \wedge \omega^{r-1}=0$ implies $(dh - \delta)\wedge\mu=0$ by the same arguments used in the proof of Lemma \ref{lemcompa}.

%\begin{equation*}
%\omega (X_{f},X_{g})=dg(\Pi ^{\#}(df))
%\end{equation*}%
%and
%\begin{equation*}
%\omega (X^{j},Y)=0
%\end{equation*}%
%for vector fields $X^{j}$ such that $\alpha _{i}(X^{j})=\delta _{ij},$ is
%closed when valued on tangent vector fields and $\mu \wedge \omega ^{r}$ is
%a volume form on $M.$ From lemma \ref{lemcompa}, we have $d\omega \wedge \mu =0$ and $%
%\omega $ is a compatible 2-form with $\mathcal{F}$. Moreover, if $\Pi $ is
%unimodular, from (\ref{2.5.5}) one have $D_{\Omega _{\omega }}(\Pi )=\Pi
%^{\#}(\delta )$ and the first obstruction class of $\mathcal{F}$ vanishes.
\end{proof}

\begin{remark}
Given k independent functions $g_{1},...,g_{k}$ in an orientable manifold $M$
with dim $M=2r+k$; the existence of a Poisson structure having those
functions as Casimir functions has been studied in \cite{DP}. The above
theorem give us sufficient and necessary conditions for the existence of the
desired Poisson bracket.
\end{remark}

\begin{example}
Regular Poisson structures on $\mathbb{R}^{4}.$
\end{example}

Consider on $\mathbb{R}^{4}$ with global coordinates $(\mathbf{x},y),\mathbf{x}=(x_{1},x_{2},x_{3})\in\mathbb{R}^{3},$ $y\in\mathbb{R},$ a regular 2-codimension foliation $\mathcal{F}$ on $\mathbb{R}^{4}$ generated by the independent 1-forms $\alpha =A\mathrm{d}\mathbf{x}+a\,\mathrm{d}y$ and $\beta =B\mathrm{d}\mathbf{x}+b\,\mathrm{d}y$, with
\begin{eqnarray*}
(A\times B)\cdot (\nabla a-\frac{\partial A}{\partial y}) &=&(aB-bA)\cdot rot A \\
(A\times B)\cdot (\nabla b-\frac{\partial B}{\partial y}) &=&(aB-bA)\cdot rot B
\end{eqnarray*}%
Then, the foliation generated by the given forms is integrable and
 $$\omega =\frac{1}{2}\frac{bA-aB}{\Vert bA-aB\Vert ^{2}}d\mathbf{x}%
\wedge d\mathbf{x}+\frac{1}{2}\frac{A\times B}{\Vert A\times B\Vert ^{2}}d%
\mathbf{x}\wedge dy$$
is a compatible 2-form with $\mathcal{F}$ and the
Poisson tensors $\Pi $ having $\mathcal{F}$ as its characteristic foliation
take the form
\begin{equation*}
\Pi =((bA-aB)\frac{\partial }{\partial \mathbf{x}}\wedge \frac{\partial }{%
\partial \mathbf{x}}+(A\times B)\frac{\partial }{\partial \mathbf{x}}\wedge
\frac{\partial }{\partial y}),
\end{equation*}%
Moreover, $\Pi $ is a unimodular Poisson tensor if $div(A\times B)=0$ and $%
\mathbf{rot((}bA-aB)+\frac{\partial }{\partial y}(A\times B)=0.$

\bigskip

\section{Regular transversally constant Poisson structures.}
%%%%%%%%%%%%%%%%%%%%%%New%%%%%%%%%%%%%%%%%%%%%%%%%%%%%%%%%%%%%%%%%%%%%%%%%%%%%%%%%%%%%%%%%%%%%%%%%%%%%

Following Vaismann \cite{Vs}, a regular Poisson tensor is transversally
constant if its characteristic foliation has a transversal distribution $%
\mathcal{T}$ such that every local vector field $X$ in $\mathcal{T}$ which
preserves the characteristic foliation $\mathcal{F}$ is a Poisson vector
field.

Consider a regular Poisson tensor $\Pi $ with a characteristic foliation
$\mathcal{F}$ generated by $k$ independent 1-forms $\alpha_1,\ldots,\alpha_k$ with local dual vector
 fields $X^j$ and a 2-form $\omega $ compatible  with $\mathcal{F}$ such that $\mathrm{i}_{X^j}\omega =0$. 
 
Under the considerations above, we have the following criteria for regular transversally constant Poisson structures.

%\begin{theorem}
%Let be $\omega $ a 2-form compatible with $\mathcal{F}$. A vector field $Z$
%is a Poisson vector field if and only if $\ $satisfies: $a)$ $%
%i_{X_{f}}L_{Z}\mu =0$ for each smooth function $f$ and $\ b)$ $\mu \wedge
%L_{Z}\omega =0.$
%\end{theorem}

\begin{theorem}
If dual vector fields $X^{j},$ $j=1,...k$ satisfy the following conditions:
a) $i_{X_{f}}L_{X^{j}}\mu =0,$ $\forall f\in C^{\infty }(M),$ and $\
b)i_{X^{j}}\mu \wedge d\omega =0;$ then each vector field $X^{j}$ is a
Poisson vector field and the Poisson structure is transversally constant.
\end{theorem}
\begin{proof}
\ The condition $a)$ means that $X^j,$ $j=1,...k$  preserve the foliation $\mathcal{F}$
and for each Hamiltonian vector field $X_{g}$ the vector field $[X^j,X_{g}]$
is a tangent vector field to $\mathcal{F}$. From the property $%
(i_{X_{f}}\omega +df)\wedge \mu =0,$ we obtain when valued on the tangent
vector field $[X^j,X_{g}]$%
\begin{equation*}
i_{[X^j,X_{g}]}i_{X_{f}}\omega +L_{[X^j,X_{g}]}f=0.
\end{equation*}%
By item $b),$ we obtain
\begin{eqnarray*}
0 &=&i_{X_{g}}i_{X_{f}}L_{X^j}\omega
=i_{X_{g}}(L_{X^j}i_{X_{f}}+i_{[X_{f},X^j]})\omega = \\
&=&-L_{[X^j,X_{f}]}g+L_{X^j}i_{X_{g}}i_{X_{f}}\omega
+i_{[X_{g},X^j]}i_{X_{f}}\omega = \\
&=&L_{X^j}\left\{ g,f\right\} -\left\{ L_{X^j}g,f\right\} -\left\{
g,L_{X^j}f\right\}
\end{eqnarray*}%
and $X^j$ is a Poisson vector field.
\end{proof}

\begin{example}
(Dirac brackets) \bigskip Let be $(M,\omega _{0})$ a symplectic $(2r+2k)-$%
manifold and $g_{1},g_{2},...,g_{2k}$ independent smooth functions such that
the matrix function $((\Delta ^{ij}))_{i,j=1,...,2k}=((\left\{
g_{i},g_{j}\right\} _{0}))$ is invertible and $((\Delta
_{ij}))_{i,j=1,...,2k}$ its inverse matrix $\sum\limits_{j=1}^{2k}\Delta
^{ij}\Delta _{jk}=\delta _{ik}.$ Consider the regular foliation $\mathcal{F}$
generated by the k independent functions $g_{i},$ $i=1,...,2k$ and the $%
k-form$ $\mu $ with%
\begin{equation*}
\mu =dg_{1}\wedge dg_{2}\wedge \cdots \wedge dg_{2k}
\end{equation*}%
Consider a 2-form $\omega ^{DIR}\in \Lambda ^{2}(M)$ satisfying $\mu \wedge
\omega ^{DIR}=0$ given by
\begin{equation*}
\omega ^{DIR}=\omega _{0}+\frac{1}{2}\sum\limits_{i,j=1}^{2k}\Delta
_{ij}dg_{i}\wedge dg_{j}
\end{equation*}%
The 2-form $\omega ^{DIR}$ is a foliated closed form. Moreover, taking into
account that $\omega _{0}$ is non-degenerate when restricted to the leaves
of the foliation, we have that $\mu \wedge (\omega ^{DIR})^{r}=\mu \wedge
\omega _{0}^{r}$ is a volume form. Then the 2-tensor $\Pi ^{DIR}$ given by%
\begin{equation*}
i_{\Pi ^{DIR}}\mu \wedge \omega _{0}^{r}=r\mu \wedge \omega _{0}^{r-1}
\end{equation*}%
is a Poisson tensor.~After some calculations we obtain%
\begin{equation*}
\Pi ^{DIR}=\Pi _{0}-\frac{1}{2}\sum\limits_{i,j=1}^{2k}\Delta
_{ij}X_{g_{i}}\wedge X_{g_{j}}
\end{equation*}%
and%
\begin{equation*}
X_{f}^{DIR}=X_{f}-\sum\limits_{i,j=1}^{2k}\Delta _{ij}\left\{
g_{i},f\right\} _{0}X_{g_{j}}
\end{equation*}%
where $\Pi _{0}$ and $X_{g_{i}},$ $i=1,...,2k$ denotes the Poisson bracket
and the Hamiltonian vector fields for $g_{i},$ $i=1,...,2k$ with respect to
the initial symplectic structure. Taking as dual vector fields to the
generators of the foliation $dg_{i},$ $i=1,...,2k$ the vector fields%
\begin{equation*}
Z_{i}=\sum_{j=1}^{2k}\Delta _{ij}X_{g_{j}}
\end{equation*}%
we can check directly for $i=1,...,2k$ that

\begin{eqnarray*}
i_{X_{f}^{DIR}}L_{Z_{i}}\mu  &=&0 \\
\mu \wedge L_{Z_{i}}\omega ^{DIR} &=&0
\end{eqnarray*}%
Then, the $Z_{i},$ $i=1,...,2k$ are Poisson vector fields transversal to the
foliation $\mathcal{F}$ and the Dirac structure is regular transversally
constant.
\end{example}

%%%%%%%%%%%%%%%%%%%%%%%%%%%%%%%%%%%%%%%%%%%%%%%%%%%%%%%%%%%%%%%%%%%%%%%%%%%%%%%%%%%%%%%%%%%%%%%%%%%%%%%%%%

%\begin{acknowledgements}
\noindent \textbf{Acknowledgements.} The authors are grateful to the research group of Geometry and Dynamical Systems of the University of Sonora. Particularly, Thanks to Professor Yu. M. Vorobev, the Ph. D. students Jos\'e Ru\'iz, Isaac Hasse and Eduardo Velasco for fruitful discussions on the preparation of this paper. %They also would like to thank the referees for their useful suggestions that helped us to improve this work.
R. Flores-Espinoza was partially supported by the National Council of Science and Technology (CONACyT) under the Grant 178690 and M. Avenda\~{n}o-Camacho  by the
National Council of Science and Technology (CONACyT) under the Grant 219631, CB-2013-01.
%\end{acknowledgements}

\end{document}